\documentclass[mathpazo]{cicp}

\usepackage{amsmath}
\usepackage{bm}
\usepackage{fancyhdr}
\usepackage{color}
\usepackage{amsmath}
\usepackage{amsfonts}
\usepackage{dsfont}
\usepackage{mathrsfs}
\usepackage{latexsym}
\usepackage{graphicx}
\usepackage{epstopdf}
\usepackage{amssymb}
\usepackage{indentfirst}
\usepackage{subfigure}
\usepackage{booktabs}
\usepackage{cite}
\usepackage{float}
\usepackage{caption}
\usepackage{booktabs}
\usepackage{threeparttable}
\usepackage{multicol}
\usepackage{multirow}
\usepackage{geometry}
\usepackage{lineno}
\usepackage{amsbsy}
\usepackage{bm}
\usepackage{dcolumn}
\usepackage{mathrsfs}
\usepackage[title]{appendix}
\usepackage{stmaryrd}
\usepackage{tikz}
\usepackage[colorlinks=true, linkcolor=blue, citecolor=blue, urlcolor=red]{hyperref}
\usepackage{algorithmic, algorithm}
\usepackage{comment}
\tikzset{>=stealth}

\newcommand{\tol}{\mathtt{tol}}

\begin{document}
\title{Adaptive Multilevel Methods for the Maxwell Eigenvalue Problem}

\author[Q.~Liang, X.~Xu and Q.~Zhang]{Qigang Liang, Xuejun Xu\corrauth\ and Qingquan Zhang}

\address{School of Mathematical Sciences, Tongji University, Shanghai 200092, China and Key Laboratory of Intelligent Computing and Applications (Tongji University), Ministry of Education}

\email{{\tt qigang\_liang@tongji.edu.cn} (Q.~Liang), {\tt xuxj@tongji.edu.cn} (X.~Xu), {\tt 2211332@tongji.edu.cn} (Q.~Zhang)}

\begin{abstract}
In this paper, we propose an adaptive multilevel preconditioned Helmholtz-Jacobi-Davidson (PHJD) method for the Maxwell eigenvalue problem with singularities. The key idea in this work is to employ the local multilevel method for preconditioning the Jacobi-Davidson correction equation. It is 
shown that our convergence factor is quasi-optimal, which means the convergence factor is independent of mesh sizes and mesh levels provided the coarse mesh is sufficiently fine. Numerical experiments on complex domains are carried out to confirm the theoretical results and demonstrate the efficiency of the proposed method.
\end{abstract}


\maketitle

\section{Introduction}
\label{sec1}
The Maxwell eigenvalue problem, arising from electromagnetic waveguides and resonances in cavities, plays an important role in computational electromagnetism (e.g., \cite{balanis2012advanced,MR2652780,MR2009375}). The governing equations are
\begin{align*}
    \bm{\mathrm{curl}}(\mu^{-1}\bm{\mathrm{curl}} \bm{u}) =&~ \omega^2 \varepsilon \bm{u} ~~~ \text{in} ~\Omega,\\
    \bm{\mathrm{div}}(\varepsilon \bm{u}) =& ~0 ~~~~~~~~ \text{in} ~\Omega,\\
    \bm{n} \times \bm{u} =& ~0 ~~~~~~~~ \text{on} ~\partial \Omega,
\end{align*}
where $\Omega \subset \mathbb{R}^3$ is a bounded polyhedral domain and $\bm{n}$ is the unit outward normal vector. The coefficients $\varepsilon$ and $\mu$ are the real relative electric permittivity and magnetic permeability, respectively, and $\omega$ is the resonant angular frequency of the electromagnetic wave for the cavity $\Omega$. Assuming the medium is homogeneous and isotropic, we may set $\varepsilon = \mu = 1$. For convenience, we denote $\lambda:= \omega^2$.

Edge finite elements, introduced by N\'ed\'elec (see \cite{MR592160}), have been widely used to solve the Maxwell eigenvalue problem. The convergence properties of edge finite element methods can be found in  \cite{MR1701792,caorsi2000convergence,fumio1989discrete,monk2001discrete,russo2009finite} and the references therein. Due to geometric irregularities such as re-entrant corners, the eigenfunction corresponding to the principal eigenvalue often exhibit pronounced singularities. Consequently, if a quasi-uniform mesh is still employed to discretize the Maxwell eigenvalue problem, a larger number of degrees of freedom will be required to achieve the desired accuracy. Adaptive edge finite element methods (AEFEM) have been proven to be highly effective in addressing the local singularities since they may accurately capture the singular behavior of eigenfunctions. In \cite{MR3918688}, Boffi introduces a residual-type a posteriori error estimator for the Maxwell eigenvalue problem and proves its reliability and efficiency. Moreover, the optimal convergence of the adaptive scheme in the case of simple eigenvalues is presented in \cite{MR3916956}. 
 
 The procedure of the AEFEM method is as follows
\begin{equation*}
    \textbf{Solve}\to \textbf{Estimate}\to \textbf{Mark}\to \textbf{Refine}.
\end{equation*}

As the number of adaptive refinement levels increases, we need to solve the large-scale discrete algebraic eigenvalue system arising from the Maxwell eigenvalue problem in the \textbf{Solve} step, which is both challenging and time-consuming. The large kernel of the $\bm{\mathrm{curl}}$ operator poses significant challenges to the design of algorithms, as the iterative process may plunge into this kernel. To prevent such an undesirable situation, Hiptmair and Neymeyr \cite{hiptmair2002multilevel} propose a projected preconditioned inverse iteration method. Multilevel preconditioned technique is applied to solve shifted Maxwell equations with Helmholtz projection handling the kernel of the $\bm{\mathrm{curl}}$ operator efficiently. Moreover, this idea has been generalized to adaptive schemes in \cite{chen2010adaptive}. A different approach, the so-called two-grid method \cite{zhou2014two}, is proposed for solving Maxwell eigenvalue problems. This method reduces the computation of Maxwell eigenvalue problems on a fine grid to solving linear indefinite Maxwell equations on the same fine grid and original eigenvalue problems on a much coarser grid. The method maintains asymptotically optimal accuracy under the assumption $h = O(H^3)$, while significantly reducing the overall computational cost. Recently, Liang and Xu \cite{MR4566815} propose a two-level preconditioned Helmholtz subspace iterative method for solving Maxwell eigenvalue problems based on domain decomposition. By ingeniously reducing the fine-scale eigenvalue problem to a sequence of parallel preconditioned systems and Helmholtz projections, the method achieves optimality for degrees of freedom, scalability for parallel subdomains and robustness for clustered eigenpairs.

The Jacobi-Davidson method \cite{sleijpen2000jacobi}, has been proven to be a powerful tool for solving eigenvalue problems. However, for large-scale discrete eigenvalue problems, the Jacobi-Davidson correction operator typically becomes ill-conditioned. Therefore, it is important to develop an efficient preconditioning technique.   The two-level preconditioned Jacobi-Davidson methods for eigenvalue problems are proposed by Zhao, Hwang and Cai \cite{MR3643572} and extended by Wang and Xu \cite{wang2018two,Wang2018OnTC} for 2$m$-th order ($m$ = 1, 2) elliptic eigenvalue problems. Recently, Guo, Liang and Xu propose a local multilevel preconditioned Jacobi-Davidson Method for elliptic eigenvalue problems on adaptive meshes in \cite{Guo2025title}. It is proved that the convergence factor of the method proposed in \cite{Guo2025title} is independent of the degrees of freedom and mesh levels. For the Maxwell eigenvalue problem, Liang and Xu \cite{2022A} propose an efficient two-level preconditioned Helmholtz-Jacobi-Davidson method, which performs well in practice and has been shown to be both optimal and scalable. Although the two-level PHJD method based on domain decomposition is proved to be effective on quasi-uniform grids, it may be limited for the Maxwell eigenvalue  problem with singularities. 

In this paper, we present an adaptive multilevel PHJD method to solve the Maxwell eigenvalue problem with singularities. In each outer iteration, we first solve a preconditioned Jacobi–Davidson correction equation. Then, we solve a Helmholtz projection system to ensure that the numerical solution is divergence-free in the discrete sense. Finally, we update the new iterative solution in the expanding Davidson subspace. To address the ill-conditioning of the Jacobi–Davidson correction equation, we design an efficient local multilevel preconditioner based on the local successive subspace correction with shifted techniques. With respect to the Laplace equation in the Helmholtz projection system, we choose the existing local multilevel preconditioner which maintains optimal and scalable performance under adaptive mesh refinement.

It is worth noting that the application of the adaptive multilevel PHJD method is far from straightforward. First, some finite element estimates for the Maxwell eigenvalue problem, that hold on quasi-uniform meshes in \cite{2022A}, are no longer valid on adaptive meshes. Second, since the Jacobi–Davidson correction equation is associated with the shifted Maxwell equations, some fundamental theoretical formulations in the local multilevel method need to be reformulated. Third, compared to elliptic eigenvalue problems, the \text{curl} operator in the Maxwell eigenvalue problem has an infinite-dimensional kernel, which makes the design and analysis of iterative methods on adaptive meshes difficult. By overcoming the challenges above, we prove that the approximate eigenvalue in the adaptive multilevel PHJD method converges uniformly to the discrete principal eigenvalue, provided that the coarse grid is sufficiently small.

The rest of this paper is organized as follows. In section 2, we introduce model problems and the finite element discretization. The adaptive multilevel PHJD method is presented in section 3.  Section 4 is dedicated to giving some useful lemmas and the detailed proof of the main result. In the final section, the 
numerical results are given to  confirm the
theoretical results and demonstrate the efficiency of the proposed method.

\section{Model problems and preliminaries}
\label{sec2}
In this section, we introduce model problems and preliminaries. For the sake of simplicity we assume that $\Omega$ is simply connected and its boundary is connected. We start with the introduction of some Hilbert spaces
\begin{align*}
H_0^1(\Omega):=& \left\{u \in H^1(\Omega)\mid u|_{\partial \Omega} = 0\right\},\\
\bm{H}(\bm{\mathrm{curl}} ; \Omega):=&
\left \{\bm{u} \in L^2(\Omega)^3 \mid \bm{\mathrm{curl}}\bm{u} \in L^2(\Omega)^3\right \}, \\
\bm{H}_0(\bm{\mathrm{curl}} ; \Omega):=&
\left \{\bm{u} \in \bm{H}(\bm{\mathrm{curl}} ; \Omega)\mid \bm{u} \times\left.\bm{n}\right|_{\partial \Omega}=\mathbf{0} \right \}, \\
\bm{H}_0(\bm{\mathrm{curl}}^0 ; \Omega):=&
\left \{\bm{u} \in \bm{H}_0(\bm{\mathrm{curl}} ; \Omega)\mid \bm{\mathrm{curl}}\bm{u}=\mathbf{0}  \right \}, \\
      \bm{H}(\mathrm{div} ; \Omega):=&
\left \{\bm{u} \in L^2(\Omega)^3 \mid \mathrm{div}\bm{u} \in L^2(\Omega)^3\right \},\\
      \bm{H}(\mathrm{div}^0 ; \Omega):=&
\left \{\bm{u} \in L^2(\Omega)^3 \mid \mathrm{div}\bm{u} =0 \right \}.
\end{align*}
The spaces $\bm{H}(\bm{\mathrm{curl}} ; \Omega)$ and $ \bm{H}(\mathrm{div} ; \Omega)$ are equipped with the norms
\[
\|\bm{u}\|_{\bm{\mathrm{curl}}} := \left( \|\bm{u}\|^2 + \|\bm{\mathrm{curl}}\, \bm{u}\|^2 \right)^{1/2}, 
\qquad
\|\bm{u}\|_{\mathrm{div}} := \left( \|\bm{u}\|^2 + \|\mathrm{div}\, \bm{u}\|^2 \right)^{1/2},
\]
where $\|\cdot\|$ is the usual $L^2$-norm induced by $L^2$-inner product $(\cdot,\cdot)$.

The variational form of the eigenvalue problem for the Maxwell system reads as follows
\begin{equation}\label{eq:2_1}
\begin{cases}
\mathrm{Find~}(\lambda,\bm{u})\in \mathbb{R}\times\bm{H}_0(\bm{\mathrm{curl}};\Omega)\cap \bm{H}(\mathrm{div}^0;\Omega)~\mathrm{with}~\bm{u}\ne 0,\\a(\bm{u},\bm{v})=\lambda b(\bm{u},\bm{v})\quad\forall ~\bm{v}\in\bm{H}_0(\bm{\mathrm{curl}};\Omega)\cap \bm{H}(\mathrm{div}^0;\Omega).
\end{cases}
\end{equation}
where $a(\bm{u},\bm{v}) := (\bm{\mathrm{curl}}\bm{u},\bm{\mathrm{curl}}\bm{v}$),
b($\bm{u},\bm{v}$) := ($\bm{u},\bm{v} $). We may define the norm $\|\cdot\|_{a}:=\sqrt{a(\cdot,\cdot)}$ in $\bm{H}_0(\bm{\mathrm{curl}};\Omega)\cap \bm{H}(\mathrm{div}^0;\Omega)$ and $\|\cdot\|_{b}:=\sqrt{b(\cdot,\cdot)}$ in $\bm{H}_0(\bm{\mathrm{curl}};\Omega)$. It is known that \eqref{eq:2_1} has a countable sequence of real eigenvalues \cite{MR1701792}. In this paper, we are interested in the case that the principal eigenvalue is simple, i.e.,
\begin{equation*}
    0 < \lambda_{1}< \lambda_{2}\leq\lambda_{3}\leq\cdots,
\end{equation*}
and corresponding eigenfunctions $ \bm{u}_1,\bm{u}_2,\bm{u}_3,\cdots$, which satisfy
\begin{equation*}
   a(\bm{u}_i,\bm{u}_j) = \lambda_ib(\bm{u}_i,\bm{u}_j)=\delta_{ij}\lambda_i.
\end{equation*}

\begin{remark}
In practical computations, it is typically observed that the principal eigenvalue of the Maxwell eigenvalue problem is simple unless the geometry or material coefficients exhibit special symmetries. Even a slight perturbation of these symmetries leads to a simple principal eigenvalue. For instance, in a rectangular cavity with three distinct side lengths, the principal eigenvalue turns out to be simple $($see  \cite{balanis2012advanced}$)$.
\end{remark}

To avoid the difficulty of enforcing the divergence-free condition, the following variational formulation is often used
\begin{equation}\label{eq:2_2}
\begin{cases}
\mathrm{Find~}(\lambda,\bm{u})\in \mathbb{R}\times\bm{H}_0(\bm{\mathrm{curl}};\Omega)~\mathrm{with}~\bm{u}\ne 0,\\a(\bm{u},\bm{v})=\lambda b(\bm{u},\bm{v})\quad\forall ~\bm{v}\in\bm{H}_0(\bm{\mathrm{curl}};\Omega).
\end{cases}
\end{equation}
Note that \eqref{eq:2_2} drops the divergence-free condition and the eigenvalue may equal to  zero (the corresponding eigenspace, i.e., $\bm{H}_0(\bm{\mathrm{curl}}^0;\Omega) = \nabla H_{0}^{1}(\Omega)$, is infinite-dimensional \cite[Theorem I.2.9]{MR851383}). Meanwhile, the eigenfunctions corresponding to $\lambda\ne 0$ remain unchanged.

 We now discuss the discretization of \eqref{eq:2_2} with the edge finite element method. Let $\{ \mathcal{T}_l, ~l=0, 1,$ $\ldots,L \} $ be a sequence of shape-regular and nested geometrically conforming simplicial triangulations, obtained by successively refining a sufficiently fine initial coarse mesh $\mathcal{T}_0$ using the longest edge bisection method. The initial mesh size is scaled such that $h_0 < 1$. For each level $l$, let $\mathcal{E}_l$ and $\mathcal{N}_l$ denote the sets of interior edges and interior nodes of $\mathcal{T}_l$, respectively. We denote by $\bm{b}_{l,E}$ the edge basis function associated with $E \in \mathcal{E}_l$ and by $\varphi_{l,p}$ the nodal basis function associated with $p \in \mathcal{N}_l$. The sets $\tilde{\mathcal{E}}_l$ and $\tilde{\mathcal{N}}_l$ involved in the local smoothing are defined as
\begin{align*}
    &\tilde{\mathcal{E}}_l=\left\{E\in\mathcal{E}_{l}:E\in\mathcal{E}_{l}\setminus\mathcal{E}_{l-1}\mathrm{~or~}E\in\mathcal{E}_{l-1}\mathrm{~but~supp}(\bm{b}_{l,E})\neq\mathrm{supp}(\bm{b}_{l-1,E})\right\},\\
    &\tilde{\mathcal{N}}_l=\left\{p\in\mathcal{N}_{l}:p\in\mathcal{N}_{l}\setminus\mathcal{N}_{l-1}\mathrm{~or~}p\in\mathcal{N}_{l-1}\mathrm{~but~supp}(\varphi_{l,p})\neq\mathrm{supp}(\varphi_{l-1,p})\right\}.
\end{align*}
For simplicity of notation, for $1\le l\le L$ we write
\begin{equation*}
    \bigcup_{i=1}^{N_l} \mathrm{span}\left\{\nabla \varphi_{l,i}\right\}=\bigcup_{p \in \tilde{\mathcal{N}}_l} \mathrm{span}\left\{\nabla \varphi_{l,p}\right\},~ \bigcup_{i=1}^{M_l} \mathrm{span}\left\{\bm{b}_{l,i}\right\}=\bigcup_{E \in \tilde{\mathcal{E}}_l} \mathrm{span}\left\{\bm{b}_{l,E}\right\},
\end{equation*}
where $N_l =\# \tilde{\mathcal{N}}_l$ and $M_l =\# \tilde{\mathcal{E}}_l$. We define the local subspace $W_{l,i}$ as
\begin{equation*}
    W_{l,i}=\begin{cases}\mathrm{span}\{\nabla\varphi_{l,i}\},&i=1,\ldots,N_l,\\
    \mathrm{span}\{\bm{b}_{l,j}\},&i=N_l+j, j=1,\ldots,M_l.\end{cases}
\end{equation*}

 On each mesh $\mathcal{T}_l$, we introduce the lowest order edge element space of the first family \cite{MR592160} and the corresponding discrete divergence-free space,
\begin{align*}
&V_l=\{\bm{u}_l\in \bm{H}_0(\bm{\mathrm{curl}};\Omega)\big|~ \bm{u}_l|_K=\bm{a}_K+\bm{b}_K\times \bm{x},~\forall~ K\in\mathcal{T}_l\},\\
&X_l=\{\bm{u}_l\in V_l\big|~ b(\bm{u}_l,\nabla p)=0, ~\forall~ p\in S_l\},
\end{align*}
where $S_l$ denotes the continuous piecewise linear polynomial space with vanishing trace on $\partial \Omega$ associated with $\mathcal{T}_l$. 

The nested structure of the meshes induces a sequence of nested edge element spaces: $V_0 \subset V_1 \subset V_2 \subset \cdots\subset V_L$. The finite element discretization of \eqref{eq:2_2} is
\begin{equation}\label{eq:2_3}
\begin{cases}
\mathrm{Find~}(\lambda_{l},\bm{u}_{l})\in \mathbb{R}\times V_l~\mathrm{with} ~\bm{u}_{l}\ne0,\\a(\bm{u}_{l},\bm{v}_l)=\lambda_{l}b(\bm{u}_{l},\bm{v}_l)\quad\forall ~\bm{v}_l\in V_l.
\end{cases}
\end{equation}
In this setting, the eigenspace corresponding to the zero eigenvalue coincides with $\nabla S_l$ (see \cite[Theorem 2.5]{MR1654571}). Moreover, for $\lambda_l \neq 0$, the corresponding eigenfunction $\bm{u}_l$ satisfies the discrete divergence-free constraint, i.e., $\bm{u}_l \in X_l$. We denote by $(\lambda_{i,l},\bm{u}_{i,l})$ the discrete eigenpairs corresponding to nonzero eigenvalues, which are ordered as
\begin{equation*}
    0 < \lambda_{1,l}< \lambda_{2,l}\leq\cdots\leq\lambda_{i,l}\leq\cdots\leq\lambda_{n_l,l},
\end{equation*}
since the principle eigenvalue is simple, in view of the a priori result (Theorem \ref{thm:2_1}), we know that the discrete principal eigenvalue is also simple.

For $0 \le l \le L$, we define the level $l$ operator $A_l : V_l \to V_l $ such that
\begin{equation*}
b(A_l\bm{w}_l,\bm{v}_l)=a(\bm{w}_l,\bm{v}_l)\quad\forall ~\bm{v}_l\in V_l.
\end{equation*}
We also define the local operator $A_{l,i}:W_{l,i}\to W_{l,i} ~(i = N_l+j,j=1,\ldots,M_l)$ such that
\begin{equation*}
b(A_{l,i}\bm{w}_{l,i},\bm{v}_{l,i})=a(\bm{w}_{l,i},\bm{v}_{l,i})\quad\forall ~\bm{v}_{l,i}\in W_{l,i}.
\end{equation*}
Denote by $\Omega_{l,i}$ the support of the basis function $\bm{b}_{l,i}$ and by $h_{l,i}$ the diameter of $\Omega_{l,i}$. Using the scaling argument and the Poincar\'e inequality, we may show that the minimum eigenvalue of $A_{l,i}$ satisfies
\begin{equation}\label{eq:2_4}
\lambda_{min}(A_{l,i})\ge O(h_{l,i}^{-2})\ge O(h_0^{-2}).
\end{equation}

It is known that the following spatial decomposition properties are valid
\begin{align}
    \bm{H}_{0}(\bm{\mathrm{curl}};\Omega)&=\nabla H_{0}^{1}(\Omega)\oplus W_1\oplus W_2,\label{de_1}\\
    V_{l}&=\nabla S_{l}\oplus W_1^l\oplus W_2^l,\label{de_2}
\end{align}
where $W_1 = \mathrm{span}\left \{  \bm{u}_1\right \} $, $W_1^l = \mathrm{span}\left \{\bm{u}_{1,l}  \right \}$ and the notation $\oplus$ denotes the $b(\cdot,\cdot)$-orthogonal (also $a(\cdot,\cdot)$-orthogonal) direct sum. Denote by $Q_l:V_L\to V_l$ the $b(\cdot,\cdot)$-orthogonal projection onto $V_l$, by $K_0^l:V_l\to \nabla S_{l}$, $Q_1^l:V_l\to W_1^l$ and $Q_2^l:V_l\to W_2^l$ the $b(\cdot,\cdot)$-orthogonal projections onto the respective subspaces of $V_l$, and by $Q_{l,i} :V_L\to W_{l,i}$ $(0\le l\le L,~N_l+1\le i\le N_l+M_l)$ the $b(\cdot,\cdot)$-orthogonal projection onto the discrete divergence-free local subspace. 

The following a priori error estimate, (see \cite[Theorem 3]{MR1804657} and \cite[Theorem 2]{MR3916956}), is useful in our convergence analysis. Throughout this paper, unless otherwise stated, the positive constant $C$ is independent of the number of mesh levels and the degrees of freedom in the finite element space.

\begin{theorem}\label{thm:2_1}
    Let $(\lambda_{1,l},\bm{u}_{1,l})$ $(0\le l\le L)$ be the first eigenpair of \eqref{eq:2_3} with $\|\bm{u}_{1,l}\|_{\bm{\mathrm{curl}}} = 1\  (\|\bm{u}_{1,l}\|_b = 1)$. Then there exists an eigenpair $(\lambda_{1},\bm{u}_{1})$ $(\bm{u}_{1} \in W_1)$ with $\|\bm{u}_1\|_{\bm{\mathrm{curl}}} = 1\ (\|\bm{u}_{1}\|_b = 1)$ such that
\begin{align*}
|\lambda_1-\lambda_{1,l}| \leq& ~C\delta_l^2(\lambda_1),  \\
  \|\bm{u}_1-\bm{u}_{1,l}\|_{\bm{\mathrm{curl}}}  \leq C\delta_l(\lambda_1) ~&(\|\bm{u}_1-\bm{u}_{1,l}\|_b \leq C\delta_l(\lambda_1)),
\end{align*}
where $\delta_l(\lambda_1)$ denotes the gap between $W_1$ and $W_1^l$, namely
\begin{equation*}
\delta_l\left(\lambda_1\right):=\sup_{\bm{w}\in W_1,\|\bm{w}\|_{\bm{\mathrm{curl}}}=1}\inf_{\bm{v}\in W_1^l}\left\|\bm{w}-\bm{v}\right\|_{\bm{\mathrm{curl}}}.
\end{equation*}
\end{theorem}

\begin{remark}
    The error estimate $\delta_l(\lambda_1) \le C h_0^s$ holds for all levels $0 \le l \le L$, with $1/2 < s \le 1$ (see \cite[Theorem 4.4]{MR1701792}).
\end{remark}

\section{The adaptive multilevel PHJD method}
\label{sec3}
In this section, we propose an adaptive multilevel PHJD method to solve the principal eigenpair of the Maxwell eigenvalue problem. Given the approximate solution $(\lambda_1^{L-1},\bm{u}_1^{L-1})$ at level $L-1$, we proceed to compute the eigenpair at level $L$ by employing the adaptive multilevel PHJD algorithm.

Denote by $\lambda^j$ the eigenvalue obtained at the $j$-th iteration of the adaptive multilevel PHJD algorithm. We introduce some simplified operators for the following discussion: $\tilde{A}_0^j = A_0-\lambda^jI, \tilde{A}_L^j = A_L-\lambda^jI, \tilde{A}_{l,i}^j = A_{l,i}-\lambda^jI~(l = 1,\ldots,L, i = N_l+1,\ldots,N_l+M_l)$, which are defined by
\begin{align*}
    b(\tilde{A}_0^j\bm{u}_0,\bm{v}_0) =& (\bm{u}_0,\bm{v}_0)_{E^j}\quad\forall \bm{u}_0,\bm{v}_0\in V_0, \\
        b(\tilde{A}_L^j\bm{u}_L,\bm{v}_L) =& (\bm{u}_L,\bm{v}_L)_{E^j}\quad\forall \bm{u}_L,\bm{v}_L\in V_L, \\
            b(\tilde{A}_{l,i}^j\bm{u}_{l,i},\bm{v}_{l,i}) =& (\bm{u}_{l,i},\bm{v}_{l,i})_{E^j}\quad\forall \bm{u}_{l,i},\bm{v}_{l,i}\in W_{l,i}, 
\end{align*}
where $(\cdot,\cdot)_{E^j}:= a(\cdot,\cdot)-\lambda^jb(\cdot,\cdot)$ and $I$ is the identity operator. Similarly, we can define $(\cdot,\cdot)_{E_L} := a(\cdot,\cdot)-\lambda_{1,L}b(\cdot,\cdot)$. These bilinear forms define inner products on the local subspaces $W_2^l$ and $W_{l,i}$ for $0\le l\le L$ and $N_l+1 \le i\le N_l+M_l$. The validity of the inner product on $W_2^l$ follows from the inequality $\lambda^j < \lambda_{2,L}$, while on $W_{l,i}$, it is ensured by choosing a sufficiently small $h_0$, as indicated in \eqref{eq:2_4}. We denote by $\|\cdot\|_{E^j}$ the norm induced by $(\cdot,\cdot)_{E^j}$ on $W_2^l$ and $W_{l,i}$.

\begin{remark}
 It is crucial to note that $\|\cdot\|_{E^j}$ is equivalent to $\|\cdot\|_{a}$ in $W_2^l$ and $W_{l,i}$, for $0\le l\le L$, $N_l+1 \le i\le N_l+M_l$. It suffices to provide the proof for $W_2^l$, as the argument for $W_{l,i}$ is analogous. First, we observe that $\|\bm{u}\|_{E^j} \le \|\bm{u}\|_a$ holds for all $\bm{u} \in W_2^l$. Conversely, we have
    $a(\bm{u},\bm{u}) = \|\bm{u}\|_{E^j}^2 + \lambda^j b(\bm{u},\bm{u}) \le \beta(\lambda_{2,l}) \|\bm{u}\|_{E^j}^2$,
where $\beta(\lambda) := 1 + \lambda^j(\lambda - \lambda^j)^{-1}$.
\end{remark}

\subsection{Local multilevel preconditioner}
To efficiently solve the Jacobi-Davidson correction equation in the adaptive multilevel PHJD method, we develop a local multilevel preconditioner.
\begin{definition}[Local multilevel preconditioner $B_L^j$]
    Given $\bm{g}\in V_L$, the operator $B_L^j:V_L\to V_L$ is defined by the following steps
    \vspace{-0.3cm}
    \begin{itemize}
        \item \textbf{Step 1.} Set $\bm{x}_0 = 0$.
        \item \textbf{Step 2.} For $l = 0$ to $L$, update the iterate by
        \[
            \bm{x}_{l+1} = \bm{x}_l + R_l^j Q_l\bigl(\bm{g} - \tilde{A}_L^j \bm{x}_l\bigr).
        \]
        \item \textbf{Step 3.} Set $B_L^j \bm{g} := \bm{x}_{L+1}$.
    \end{itemize}
\end{definition}
\vspace{-0.3cm}
For the smoother, we consider only the local edge smoother of Jacobi-type $R_l^j \colon V_l \to V_l$ which performs Jacobi relaxations only on the edges in $\tilde{\mathcal{E}}_l$.

\begin{equation}\label{eq:3_1}
        R_l^j\colon=
    \begin{cases}
    (\tilde{A}_0^j)^{-1}Q_2^0Q_0&\quad l=0,\\
    \gamma\sum_{i=N_l+1}^{N_l+M_l}(\tilde{A}_{l,i}^j)^{-1}Q_{l,i}&\quad1\le l\le L,
    \end{cases}
\end{equation}
where $\gamma $ is a suitably chosen positive scaling factor $(0 < \gamma  < 1)$.  For $l = 0$, the operator $Q_2^0$ is introduced to ensure that $(\tilde{A}_0^j )^{-1}$ is well-defined, owing to the fact that $\lambda^j<\lambda_{2,0}$.

\begin{remark}
Although we only analyze the local Jacobi-type smoother in the convergence analysis in this paper, numerical simulations and theoretical results may be extended to the local Gauss-Seidel smoother.
\end{remark}

Based on the multilevel framework described in \cite{MR1247694}, the error operator is defined as 
\[E_{L}^{j}=I-B_L^j\tilde{A}_L^j = (I-T_{L}^{j})\cdots(I-T_{1}^{j})(I-T_{0}^{j}),\] 
where $T_l^j = R_l^jQ_l\tilde{A}_L^j,~l = 0,\ldots,L$. 

Due to the fact that $R^j_l$ is symmetric with respect to $b(\cdot, \cdot)$, it is easy to see that $\tilde{E}_{L}^{j}:=(I-T_{0}^{j})(I-T_{1}^{j})\cdots(I-T_{L}^{j})$ is the adjoint of $E_L^j$ with respect to $(\cdot,\cdot)_{E^j}$. Further, it follows from $E_L^j = (I-T_L^j)E_{L-1}^j$ that
\begin{equation}\label{eq:3_2}
    I-E_L^j=\sum_{l=0}^LT_l^jE_{l-1}^j.
\end{equation}

\subsection{Algorithm}
We propose the following adaptive multilevel PHJD algorithm for the Maxwell eigenvalue problem.

\begin{algorithm}[H]
\caption{The adaptive multilevel PHJD algorithm}
\label{alg:1}
1. Initialization:
 let $ \lambda^0 = \lambda_1^{L-1},~\bm{u}^0 = \bm{u}_1^{L-1}$ and set $W^0 = \mathrm{span}\left \{ \bm{u}^0 \right \} $.

2. Solve the Helmholtz projection system to get $\tilde{\bm{u}}^1 = (I-K_0^L)\bm{u}^0$,

\quad \text{ }let $\bm{u}^1 = \tilde{\bm{u}}^1/\left \| \tilde{\bm{u}}^1 \right \| _b$, $\lambda^1 = Rq(\bm{u}^1) $ and set $W^1 = \mathrm{span}\left \{ \bm{u}^1 \right \}$.

3. Set $j =1$, $\bm{r}^j = \lambda^j\bm{u}^j-A_L\bm{u}^j$.

4. If $\|\bm{r}^j\|_b < \tol$, output $\lambda^j$ and $\bm{u}^j$; else, do the following:

(a) Solve preconditioned Jacobi-Davidson correction equation, i.e., 
\begin{equation*}
    \bm{e}^{j+1} = (I-Q_{U^j})B_L^j\bm{r}^j,
\end{equation*}
\quad \text{ }\text{ }where $Q_{U^j}: V_L \to U^j:= \mathrm{span}\left \{ \bm{u}^j \right \}$ is the $b(\cdot,\cdot)$-orthogonal projection.

(b) Solve the Helmholtz projection system to get $\bm{t}^{j+1} = (I-K_0^L)\bm{e}^{j+1}$.

(c) Solve the first eigenpair in $W^{j+1}$
\[
a(\bm{u}^{j+1},\bm{v}) = \lambda^{j+1}b(\bm{u}^{j+1},\bm{v})\quad\forall \ \bm{v}\in W^{j+1},~\|\bm{u}^{j+1}\|_b=1,
\]
\quad \text{ }\text{ }where $W^{j+1} = W^j+\mathrm{span}\left \{ \bm{t}^{j+1} \right \}$.

(d) Set $j= j+1$, $\bm{r}^j = \lambda^j\bm{u}^j-A_L\bm{u}^j$, go to step 4.
\end{algorithm}

\begin{remark}
    If the topology of either $\Omega$ or $\partial \Omega$ is more complex, the kernel of curl operator is spanned by grad $H_0^1(\Omega)$ together with a low-dimensional subspace of  harmonic vector fields. In this case, the algorithm can be modified following \cite{hiptmair2002multilevel} to handle the kernel of the curl operator. To demonstrate that our algorithm can be extended to computational domains or their boundaries with non-trivial topology, we present several numerical examples, see Example \ref{ex_4} and Example \ref{ex_6}.
\end{remark}

\section{Convergence analysis}
\label{sec4}
In this section, we derive the convergence of the proposed adaptive multilevel PHJD method for the Maxwell eigenvalue problem. We begin by stating the main convergence result.

\begin{theorem}\label{thm:4_1}
     For sufficiently small $h_0$ such that $0<\theta<1$, we have the following convergence estimates
    \begin{equation}\label{eq:4_1}
        \|\bm{e}_L^{j+1}\|_{E^j}\leq\begin{Bmatrix}1-\alpha^j(1-\theta)+Ch_0^s\end{Bmatrix}\|\bm{e}
        _L^j\|_{E^j},
    \end{equation}
    and
    \begin{equation}\label{eq:4_2}
        \lambda^{j+1}-\lambda_{1,L}\leq\left\{1-\alpha^j(1-\theta)+Ch_0^s\right\}^2(\lambda^j-\lambda_{1,L}),
    \end{equation}
    where $\bm{e}^{j}_L = -Q_2^L\bm{u}^j$, $0<\alpha^j<\frac{1}{1-\theta}$ and $s>\frac{1}{2}$. 
\end{theorem}
\begin{remark}
    In Theorem \ref{thm:4_1}, we choose sufficiently small $h_0$ to ensure that $0<\theta<1$. In fact, 
    \begin{equation*}
        \theta = \sqrt{1-\frac{(2-\omega)\left(1-C \delta_0(\lambda_1)\right)^2}{C}+h_0^{2s}}\to \sqrt{1-\frac{(2-\omega)}{C}}\in(0,1), as~h_0\to 0.
    \end{equation*}
    where $\omega \in(0,2)$ depends on the positive scaling factor of Jacobi smoother and the shape regularity of the meshes.
\end{remark}

\subsection{A rigorous proof of the convergence result.}
In this subsection, we present a rigorous proof of the convergence result. The key idea is to introduce an auxiliary error operator and decompose it into a principal error term and a remainder term. Each term will be estimated individually, and the convergence result follows from the relation between the auxiliary and actual error terms.

To construct the auxiliary error operator, we introduce an auxiliary function defined by
\begin{equation*}
    \tilde{\bm{u}}^{j+1}:=\bm{u}^j+\alpha^j\bm{t}_L^{j+1}\in U^j+\mathrm{span}\{\bm{t}_L^{j+1}\}\subset W^{j+1},
\end{equation*}
where $\alpha_j$ is an undetermined parameter. Using the definition of 
$\bm{t}_L^{j+1}$ in Algorithm \ref{alg:1}, it follows that
\begin{equation}\label{eq:4_7}
\tilde{\bm{u}}^{j+1}=\bm{u}^j+\alpha^j(I-K_0^L)(I-Q_{U^j})B_L^j\bm{r}^j.
\end{equation}

Combining \eqref{eq:4_7}, the definition of residual vector $\bm{r}^j$ and the partitions of the identity operator in $V_L$, we obtain
\begin{align}\label{eq:4_8}
\begin{aligned}
       \tilde{\bm{e}}_L^{j+1} &= -Q_2^L\tilde{\bm{u}}^{j+1}\\
     &=  \bm{e}_L^{j}-\alpha^jQ_2^LB_L^j\bm{r}^j+\alpha^jQ_2^LQ_{U^j}B_L^j\bm{r}^j\\
     &= \bm{e}_{L}^{j}-\alpha^{j}Q_{2}^{L}B_{L}^{j}(\lambda^{j}I-A_{L})(K_0^L\bm{u}^{j}+Q_{1}^{L}\bm{u}^{j}-\bm{e}_{L}^{j})+\alpha^{j}Q_{2}^{L}Q_{U^{j}}B_{L}^{j}\bm{r}^{j}\\
     &=  \{\bm{e}_{L}^{j}+\alpha^{j}Q_{2}^{L}B_{L}^{j}(\lambda^{j}I-A_{L})\bm{e}_{L}^{j}\}+\alpha^{j}\{Q_{2}^{L}B_{L}^{j}(A_{L}-\lambda^{j}I)Q_{1}^{L}\bm{u}^{j}\\
     & \quad +Q_{2}^{L}Q_{Uj}B_{L}^{j}\bm{r}^{j}\}\\
     &=: I_{1}^{j}+I_{2}^{j}, 
\end{aligned}
\end{align}
where $I_{1}^{j}$ is the principal error term and $I_{2}^{j}$ represents the remainder term.

For the estimate of $I_{1}^{j}$, we define the principal error operator $G_L^j:W_2^L\to W_2^L$ as follows
\begin{equation*}
    G_L^j:=I+\alpha^jQ_2^LB_L^j(\lambda^jI-A_L)=I-\alpha^jQ_2^LB_L^j\tilde{A}_L^j.
\end{equation*}

An upper bound for the norm of principal error operator is given in the following lemma.
\begin{lemma}\label{lem:4_9}
    Under the same assumption as in Theorem \ref{thm:4_1}, it holds that
    \begin{equation*}
        \|G_L^j\bm{v}\|_{E^j}\leq\begin{pmatrix}1-\alpha^j(1-\theta)\end{pmatrix}\|\bm{v}\|_{E^j}\quad\forall \bm{v}\in W_2^L,
    \end{equation*}
    where the coefficient $\alpha^j$ satisfies $0<\alpha^j<\frac{1}{1-\theta}$.
\end{lemma}

\begin{proof}
By applying the definitions of $G_L^j$ and $E_L^j$ together with the triangle inequality, we have
    \begin{equation}\label{eq:4_9}
        \|G_L^j\bm{v}\|_{E^j}=\|(I-\alpha^jQ_2^L)\bm{v}+\alpha^jQ_2^LE_L^j\bm{v}\|_{E^j}\leq(1-\alpha^j)\|\bm{v}\|_{E^j}+\alpha^j\|Q_2^LE_L^j\bm{v}\|_{E^j}.
    \end{equation}
    To bound $\|Q_2^LE_L^j\bm{v}\|_{E^j}$, by the discrete Helmholtz decomposition and the definition of $(\cdot,\cdot)_{E^j}$, we obtain
    \begin{align}\label{eq:4_10}
(Q_{2}^{L}E_{L}^{j}\bm{v},Q_{2}^{L}E_{L}^{j}\bm{v})_{E^{j}}& =(E_{L}^{j}\bm{v},E_{L}^{j}\bm{v})_{E^{j}}-(Q_{1}^{L}E_{L}^{j}\bm{v},Q_{1}^{L}E_{L}^{j}\bm{v})_{E^{j}}\nonumber\\&\quad-(K_{0}^{L}E_{L}^{j}\bm{v},K_{0}^{L}E_{L}^{j}\bm{v})_{E^{j}} \nonumber\\
& \le (E_{L}^{j}\bm{v},E_{L}^{j}\bm{v})_{E^{j}}+(\lambda^{j}-\lambda_{1,L})b(E_{L}^{j}\bm{v},E_{L}^{j}\bm{v})\nonumber\\&\quad+\lambda^{j}b(K_{0}^{L}E_{L}^{j}\bm{v},K_{0}^{L}E_{L}^{j}\bm{v}).
\end{align}
We begin by estimating the first term in \eqref{eq:4_10}. Let $T^j = \sum_{l = 0}^{L} T_l^j$. Following the argument of Lemma 5 in \cite{Guo2025title}, we derive the following lower bound of $T^j$ in $W_2^L$
\begin{equation}\label{eq:lem:4_5}
            \|\bm{v}\|_{E^j}^2\leq\frac{C}{\left(1-C\delta_0(\lambda_1)\right)^2}(T^j\bm{v},\bm{v})_{E^j}\quad\forall \bm{v}\in W_2^L.
\end{equation}
Moreover, the classical theory of multilevel methods \cite{MR1247694} implies the following estimates
\begin{equation}\label{eq:lem:4_6_1}
    (T^j\bm{v}_L,\bm{v}_L)_{E^j}\leq C\sum_{l=0}^L(T_l^jE_{l-1}^j\bm{v}_L,E_{l-1}^j\bm{v}_L)_{E^j}\quad\forall \bm{v}_L\in V_L,
\end{equation}
and
\begin{equation}\label{eq:lem:4_6_2}
    (2-\omega)\sum_{l=0}^{L}(T_{l}^{j}E_{l-1}^{j}\bm{v}_{L},E_{l-1}^{j}\bm{v}_{L})_{E^{j}}\leq(\bm{v}_{L},\bm{v}_{L})_{E^{j}}-(E_{L}^{j}\bm{v}_{L},E_{L}^{j}\bm{v}_{L})_{E^{j}}\quad\forall \bm{v}_{L}\in V_{L},
\end{equation}
where $\omega= max_{l=0,\ldots,L} ~\omega_l~ (0 < \omega_l < 2)$ is determined by the damping parameter of the Jacobi smoother and the shape regularity of the meshes.
Combining \eqref{eq:lem:4_5}, \eqref{eq:lem:4_6_1}, and \eqref{eq:lem:4_6_2}, we obtain the estimate
\begin{equation}\label{eq:4_11}
    (E_L^j\bm{v},E_L^j\bm{v})_{E^j}\leq\left\{1-\frac{(2-\omega)\left(1-C_2 \delta_0(\lambda_1)\right)^2}{C_1}\right\}\|\bm{v}\|_{E^j}^2.
\end{equation}

Next, the second term in \eqref{eq:4_10} can be bounded using the estimates $\lambda^{j}- \lambda_{1,L} \leq Ch_0^{2s}$ and $\|E_L^j\bm{v}\|_b \leq C\|\bm{v}\|_a \leq C\|\bm{v}\|_{E^j}$. The first inequality follows from mathematical induction and Lemma~4.3 in \cite{2022A}, while the second is derived using the Poincar\'e inequality and arguments similar to Lemma 8 in \cite{Guo2025title}.

Then we estimate the remaining term in \eqref{eq:4_10}. Applying \eqref{eq:3_2}, we get
\begin{align}\label{eq:4_18}
    b(K_{0}^{L}E_{L}^{j}\bm{v},&K_{0}^{L}E_{L}^{j}\bm{v}) = ~b(K_{0}^{L}E_{L}^{j}\bm{v},E_{L}^{j}\bm{v})=b(K_{0}^{L}E_{L}^{j}\bm{v},E_{L}^{j}\bm{v}-\bm{v})\nonumber\\
    =&-b(K_0^LE_L^j\bm{v},T_0^j\bm{v})-b(K_0^LE_L^j\bm{v},\sum_{l=1}^LT_l^jE_{l-1}^j\bm{v})).
    \end{align}
It follows from the definition of $R_0^j$ in \eqref{eq:3_1} that $T_0^j\bm{v}$ is discretely divergence-free. Let $\bm{w}_0$ denote the divergence-free component of $T_0^j\bm{v}$ in the Helmholtz decomposition in $\bm{H}_0(\bm{\mathrm{curl}};\Omega)$. Then, by Lemma 7.6 in \cite{monk2003finite}, we have
   \begin{equation}\label{eq:4_12}
           \|\bm{w}_0-T_0^j\bm{v}\|_b \le Ch_0^s\|\bm{\mathrm{curl}} T_0^j\bm{v}\|_b,
   \end{equation}
   where C is a constant independent of the mesh size and $\frac{1}{2}<s\le 1$. 
    Since $\bm{w}_0\in W_1 \oplus W_2$, combining \eqref{eq:4_18} and \eqref{eq:4_12} we obtain
    \begin{align}\label{eq:4_19}
b(K_{0}^{L}E_{L}^{j}\bm{v},K_{0}^{L}E_{L}^{j}\bm{v})= &-b(K_0^LE_L^j\bm{v},T_0^j\bm{v}-\bm{w}_0)-b(K_0^LE_L^j\bm{v},\sum_{l=1}^LT_l^jE_{l-1}^j\bm{v}),\nonumber\\
        \le &~~\|K_0^LE_L^j\bm{v}\|_b(Ch_0^{s}\|T_0^j\bm{v}\|_{E^j}+\|\sum_{l=1}^LT_l^jE_{l-1}^j\bm{v}\|_b).
\end{align}
Since $R_0^j$ is symmetric with respect to $b(\cdot, \cdot)$, we have 
\begin{equation}\label{eq:4_13}
    (T_0^j\bm{v}_L,\bm{\omega}_2^0)_{E^j} = (\bm{v}_L,\bm{\omega}_2^0)_{E^j}~~~\forall \bm{v}_L\in V_L,\bm{\omega}_2^0\in W_2^0.
\end{equation}
It follows from \eqref{eq:4_13} that
\begin{align}\label{eq:4_20}
    \|T_0^j\bm{v}\|_{E^j}^2 = (T_0^j\bm{v},\bm{v})_{E^j} \le \|T_0^j\bm{v}\|_{E^j}\|\bm{v}\|_{E^j},
\end{align}
To bound the second term in \eqref{eq:4_19}, we employ \eqref{eq:lem:4_6_2} and the Poincar\'e inequality on the local discrete divergence-free subspace. This yields the estimate 
\begin{equation}\label{eq:lem4_7}
     \|\sum_{l=1}^LT_l^jE_{l-1}^j\bm{v}\|_b\le Ch_0((\bm{v},\bm{v})_{E^j}+\lambda^jb(E_L^j\bm{v},E_L^j\bm{v}))^{\frac{1}{2}}\quad \forall \bm{v}\in W_1^L\oplus W_2^L.
\end{equation}
Using the estimate $\|E_L^j\bm{v}\|_b \leq C\|\bm{v}\|_{E^j}$ , together with \eqref{eq:4_19}, \eqref{eq:4_20} and \eqref{eq:lem4_7}, we arrive at
\begin{align}\label{eq:4_14}
        \|K_0^LE_L^j\bm{v}\|_b\le& ~Ch_0^{s}\|\bm{v}\|_{E^j}+Ch_0(\|\bm{v}\|_{E^j}^2+\lambda^j\|E_L^j\bm{v}\|_b^2)^{\frac{1}{2}}
        \le~ Ch_0^{s}\|\bm{v}\|_{E^j}.
\end{align}

Finally, combining \eqref{eq:4_10}, \eqref{eq:4_11} and \eqref{eq:4_14}, we obtain
\begin{equation}\label{eq:4_15}
    \|Q_2^LE_L^j\bm{v}\|_{E^j}\le \sqrt{1-\frac{(2-\omega)\left(1-C \delta_0(\lambda_1)\right)^2}{C}+Ch_0^{2s}}\|\bm{v}\|_{E^j}.
\end{equation}
Let $\theta = \sqrt{1-\frac{(2-\omega)\left(1-C \delta_0(\lambda_1)\right)^2}{C}+Ch_0^{2s}}$. Substituting \eqref{eq:4_15} into \eqref{eq:4_9} completes the proof of this lemma.
\end{proof}

To estimate $I_2^j$, we present the following lemma
\begin{lemma}\label{lem:4_10}
    Under the same assumption as in Theorem \ref{thm:4_1}, it holds that
    \begin{equation*}
        \|I_{2}^j\|_{E^j}\le Ch_0^s\|\bm{e}_L^j\|_{E^j}.
    \end{equation*}
\end{lemma}

\begin{proof}
For the readability, we put the proof of this lemma in the appendix.
\end{proof}

Now we are in a position to present a proof of the main result.

\textbf{Proof of Theorem \ref{thm:4_1}}:
    By \eqref{eq:4_8}, Lemma \ref{lem:4_9} and Lemma \ref{lem:4_10}, we have 
    \begin{eqnarray}\label{eq:4_16}
        \|\tilde{\bm{e}}_L^{j+1}\|_{E^j} &\le& \|I_1^j\|_{E^j} + \|I_2^j\|_{E^j}\nonumber\\
        &\le& (1-\alpha^j(1-\theta))\|\bm{e}_L^j\|_{E^j} + Ch_0^s\|\bm{e}_L^j\|_{E^j}\nonumber\\
        &\le& (1-\alpha^j(1-\theta)+Ch_0^s)\|\bm{e}_L^j\|_{E^j}.
    \end{eqnarray}
Since $U^j+\mathrm{span}\{t_L^{j+1}\}\subset W^{j+1}$, it follows that
        \begin{equation*}
            \lambda^{j+1} \le \tilde{\lambda}^{j+1} = Rq(\tilde{\bm{u}}^{j+1}).
        \end{equation*}
Then we obtain
\begin{equation}\label{eq:4_17}
    \lambda^{j+1}-\lambda_{1,L} \le \tilde{\lambda}^{j+1}-\lambda_{1,L} = \|\tilde{\bm{e}}_L^{j+1}\|_{E^j}^2+(\lambda^j-\lambda_{1,L})\|\tilde{\bm{e}}_L^{j+1}\|_b^2.
\end{equation}
Combining \eqref{eq:4_16}, \eqref{eq:4_17} and the Poincar\'e inequality, we derive
\begin{align*}
        \lambda^{j+1}-\lambda_{1,L} \le&~ (1+Ch_0^s)(1-\alpha^j(1-\theta)+Ch_0^s)\|\bm{e}_L^j\|_{E^j}^2\\ 
        \le &~ (1-\alpha^j(1-\theta)+Ch_0^s)^2\|\bm{e}_L^j\|_{E^j}^2 \\
        \le& ~(1-\alpha^j(1-\theta)+Ch_0^s)^2(\lambda^j-\lambda_{1,L}),
\end{align*}
This implies \eqref{eq:4_2}. Regarding \eqref{eq:4_1}, we observe that
\begin{equation*}
    \|\bm{e}_L^{j+1}\|_{E^j}\le \sqrt{\lambda^{j+1}-\lambda_{1,L}} \le (1-\alpha^j(1-\theta)+C h_0^s)\|\bm{e}_L^j\|_{E^j},
\end{equation*}
which completes the proof. \qed

\section{Numerical results}
\label{sec5}
In this section, we will present several numerical experiments to verify the effectiveness and robustness of Algorithm \ref{alg:1} on simply connected domains, multiply connected domains, and the crack domain. The adaptive mesh refinement is driven by a residual-based a posteriori error estimator (see \cite{MR3918688}). Let $(\lambda_{1,L}, u_{1,L})$ denote the finite element approximation of the eigenpair $(\lambda_1, u_1)$, satisfying the normalization conditions $\|u_{1,L}\|_b=1$ and $(u_1, u_{1,L}) > 0$. For each $K \in \mathcal{T}_L$, the a posteriori error indicator is defined by
\begin{align*}
        \mu_L^2(\bm{u}_{1,L},K) :=&~  h_K^2\left\|\bm{u}_{1,L}-\bm{\mathrm{curl}}\left(\bm{\mathrm{curl}} \bm{u}_{1,L} / \lambda_{1,L}\right)\right\|_{0, K}^2\\
        +\frac{1}{2} \sum_{F \in \mathcal{F}_{\mathrm{I}}(K)}\{h_F\| [\![&\left(\bm{\mathrm{curl}} \bm{u}_{1,L} / \lambda_{1,L}\right) \times \bm{n}]\!] \|_{0, F}^2
    +h_F\| [\![ \bm{u}_{1,L} \cdot \bm{n}]\!] \|_{0, F}^2\},
\end{align*}
where $\mathcal{F}_{\mathrm{I}}(K)$ denotes the set of interior faces of the element $K$, $h_F$ the diameter of $F$, $[\![\cdot]\!]$ the jump across $F$, and $\bm{n}$ the unit normal vector to $F$. Then the global a posteriori error estimate on $\mathcal{T}_L$ is defined by
\begin{equation*}
    \mu_L(\bm{u}_{1,L},\mathcal{T}_L):=\left(\sum_{K \in \mathcal{T}_L} \mu_L^2(\bm{u}_{1,L},K)\right)^{1 / 2}.
\end{equation*}

\par The stopping criterion of the algorithm is $\|\bm{r}^j\|_b<10^{-8}$. In the table summarizing the numerical results, DOF. denotes the degrees of freedom; Iters represents the number of iterations in the Jacobi-type or the Gauss-Seidel-type local multilevel smoother;  Res. denotes the residual norm at convergence (i.e., $\|\bm{r}^j\|_b$) and $\lambda_1^l$ represents the principal eigenvalue at level $l$ obtained by Algorithm \ref{alg:1}. For each of the following examples, the locally refined meshes and the corresponding solutions on a certain level are presented. The reduction of the a posteriori error is presented to confirm the efficiency and reliability of the adaptive finite element method. In addition, the quasi-optimality of the adaptive multilevel PHJD method is demonstrated by the boundedness of the number of iteration counts and the nearly linear increase of CPU time.

\begin{example}\label{ex:5_1}
    We consider the Maxwell eigenvalue problem on the Fichera corner domain $\Omega = (0, 2\pi)^3 \setminus  (\pi, 2\pi)\times(\pi, 2\pi)\times(\pi, 2\pi)$.
\end{example}

\begin{figure*}[htbp]
\centering
\subfigure[level = 1]{
\begin{minipage}[t]{0.45\linewidth} 
\centering
\includegraphics[scale=0.5]{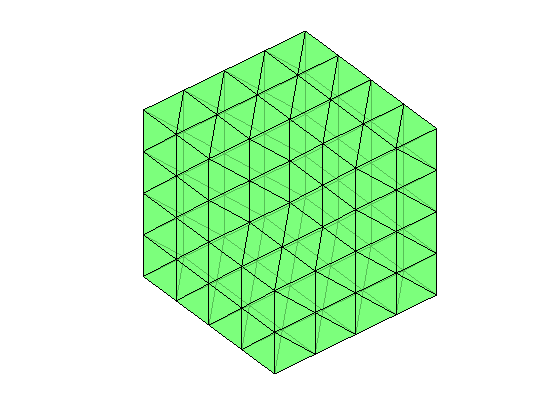}
\end{minipage}
}
\subfigure[level = 15]{
\begin{minipage}[t]{0.45\linewidth} 
\centering
\includegraphics[scale=0.5]{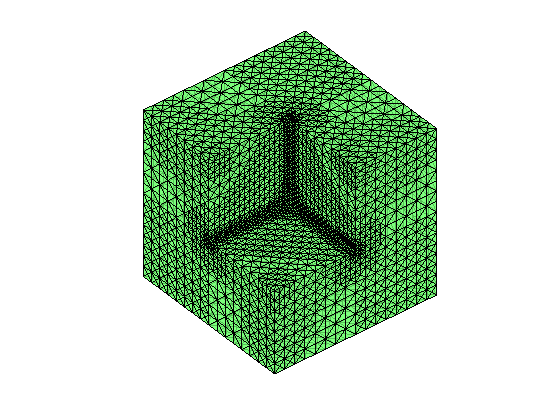}
\end{minipage}
}
\centering
\caption{The local refined mesh on adaptive levels 1 and 15}
\label{fig:5_1}
\end{figure*}

\begin{table}[H]
\centering
\begin{tabular}{|c|c|c|c|c|c|}
\hline
  Level($l$) & DOF. & Iters(Jacobi) & Iters(GS) & Res. & $\lambda_1^l$\\
  \hline
  4 & 1282 & 24 & 14 & 4.3274e-09 & 0.301528057\\
  \hline
  8 & 9698 & 24 & 14 & 5.1436e-09 & 0.317458912\\
  \hline
  11 & 38629 & 24 & 13 & 7.8428e-09 & 0.322869279\\
  \hline
  14 & 152014 & 24 & 13 & 8.6064e-09 & 0.324486372 \\
  \hline
  16 & 360002 & 23 & 13 & 5.2117e-09 & 0.325503648 \\
  \hline
  19 & 1325916 & 23 & 12 & 9.5636e-09 & 0.325931977 \\
  \hline
\end{tabular}
\caption{The numerical results on each level}\label{tab:5_1}
\end{table}

Figure \ref{fig:5_1} displays a locally refined mesh of 234,299 elements generated by the adaptive finite element algorithm, which accurately captures the locations of singularities. It can be observed from Table \ref{tab:5_1} that the adaptive multilevel PHJD algorithm converges in nearly the same number of iterations across all levels, although the number of degrees of freedom increases to 1,325,916. This indicates that the convergence rate is quasi-optimal.
\begin{figure}[htbp]
\begin{minipage}[t]{0.5\textwidth}
\centering
\includegraphics[width=\textwidth]{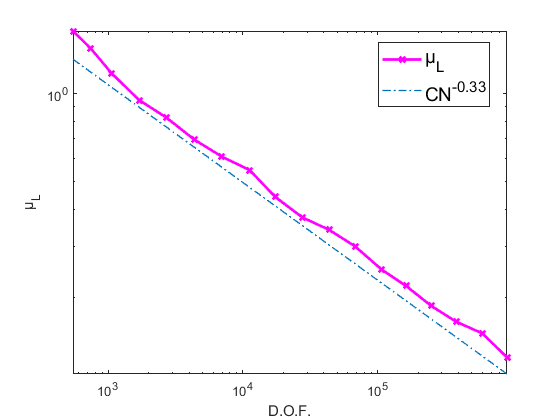}
\caption{Reduction of a posteriori error}
\label{fig:5_2}
\end{minipage}
\begin{minipage}[t]{0.5\textwidth}
\centering
\includegraphics[width=\textwidth]{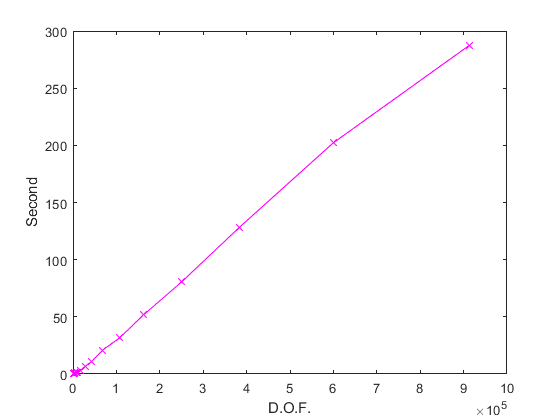}
\caption{Increasing of CPU time}
\label{fig:5_3}
\end{minipage}
\end{figure}

Figure \ref{fig:5_2} shows that the error curves obtained with uniform meshes and with adaptively refined meshes are almost parallel, which verifies the effectiveness and reliability of the a posteriori error estimator. Figure \ref{fig:5_3} illustrates that the CPU time of the adaptive multilevel PHJD iteration scales linearly with the number of degrees of freedom.

\begin{example}\label{ex_3}
    To simulate the situation of electromagnetic coils in practical applications, we consider the Maxwell eigenvalue problem on domain $\Omega = (0, 2\pi)^3 \setminus  (0.5\pi,1.5\pi)\times(0.5\pi,1.5\pi)\times(0,2\pi)$
\end{example}

Figure \ref{fig:5_7} displays a locally refined mesh of 273,714 elements generated by the adaptive finite element algorithm. Local mesh refinement is mainly concentrated in the corner regions within the solenoid interior. It can be observed from Table \ref{tab:5_3} that the adaptive multilevel PHJD algorithm converges in almost the same small number of steps. This indicates that the convergence rate is quasi-optimal, which illustrates that our adaptive multilevel PHJD method is scalable for the multiply connected domain.

Figure \ref{fig:5_8} shows that the error curves obtained with uniform meshes and with adaptively refined meshes are almost parallel. Figure \ref{fig:5_9} illustrates that the CPU time for solving the algebraic system increases roughly linearly with respect to the degrees of freedom.

\begin{figure*}[htbp]
\centering
\subfigure[level = 1]{
\begin{minipage}[t]{0.45\linewidth} 
\centering
\includegraphics[scale=0.5]{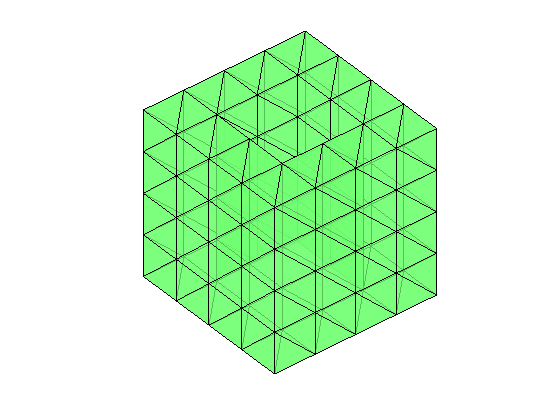} 
\end{minipage}
}
\subfigure[level = 15]{
\begin{minipage}[t]{0.45\linewidth} 
\centering
\includegraphics[scale=0.5]{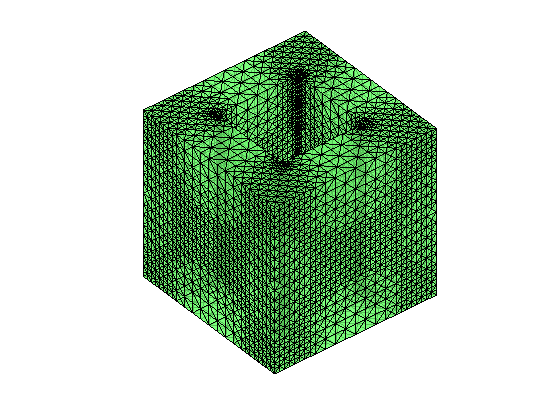}
\end{minipage}
}
\centering
\caption{The local refined mesh on adaptive levels 1 and 15}
\label{fig:5_7}
\end{figure*}

\begin{table}[H]
\centering
\begin{tabular}{|c|c|c|c|c|c|}
\hline
  Level($l$) & DOF. & Iters(Jacobi) & Iters(GS) & Res. & $\lambda_1^l$\\
  \hline
  4 & 2071 & 25 &14 & 6.6786e-09 & 0.245025619\\
  \hline
  8 & 11804 & 30 &14 & 5.5633e-09 & 0.248053377\\
  \hline
  11 & 52344 & 28 &13 & 6.6920e-09 & 0.248952103\\
  \hline
  14 & 196592 & 28 &13 & 4.9924e-09 & 0.249614806 \\
  \hline
  16 & 465860 & 27 &12 & 9.0458e-09 & 0.249753098 \\
  \hline
  18 & 1102892 & 28 &13 & 4.8462e-09 & 0.249876756 \\
  \hline
\end{tabular}
\caption{The numerical results on each level}\label{tab:5_3}
\end{table}

\begin{figure}[htbp]
\begin{minipage}[t]{0.5\textwidth}
\centering
\includegraphics[width=\textwidth]{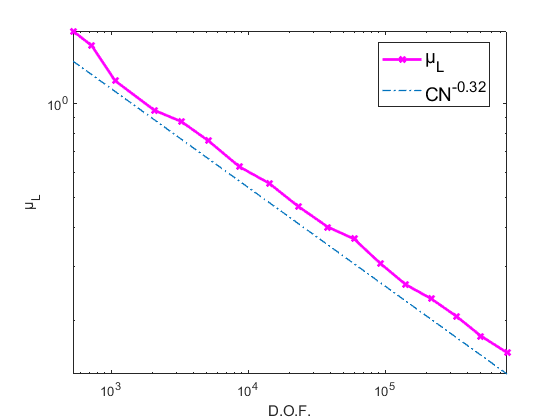}
\caption{Reduction of a posteriori error}
\label{fig:5_8}
\end{minipage}
\begin{minipage}[t]{0.5\textwidth}
\centering
\includegraphics[width=\textwidth]{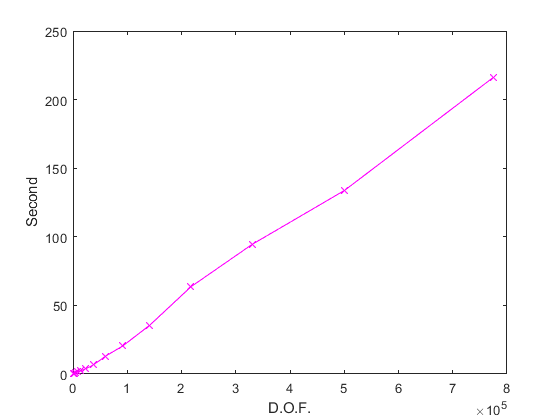}
\caption{Increasing of CPU time}
\label{fig:5_9}
\end{minipage}
\end{figure}

\par The following two examples demonstrate that our adaptive multilevel PHJD method remains efficient for domains with multiply connected boundaries.

\begin{example}\label{ex_4}
    We consider the Maxwell eigenvalue problem on cubic cavity whose boundary is multiply connected: $\Omega = (0, 2\pi)^3 \setminus  (Q_1\cup Q_2)$, where $Q_1 = (0.5\pi,\pi)\times(0.5\pi,\pi)\times(\pi,1.5\pi)$ and $Q_2 = (0.5\pi,\pi)\times(\pi,1.5\pi)\times(0.5\pi,\pi)$.
\end{example}

We modified Algorithm \ref{alg:1} by adding a step that projects the iterative solution onto the subspace orthogonal to the low-dimensional space of harmonic vector fields. Figure \ref{fig:5_13} shows a locally refined mesh of 229,019 elements generated by the adaptive finite element algorithm. We observe that the mesh is locally refined near the boundary of the cavities. Table \ref{tab:5_5} indicates that the iteration count of the adaptive multilevel PHJD algorithm remains small and stable, even as the number of degrees of freedom increases to 991,722. This indicates that the convergence rate is quasi-optimal, demonstrating that our adaptive multilevel PHJD method is scalable for the multiply connected domain.

\begin{figure*}[htbp]
\centering
\subfigure[level = 1]{
\begin{minipage}[t]{0.45\linewidth} 
\centering
\includegraphics[scale=0.5]{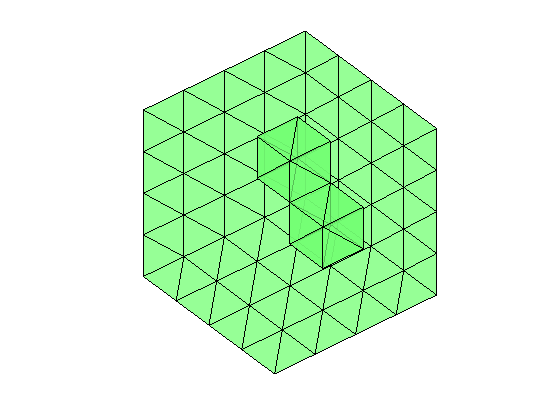} 
\end{minipage}
}
\subfigure[level = 15]{
\begin{minipage}[t]{0.45\linewidth} 
\centering
\includegraphics[scale=0.5]{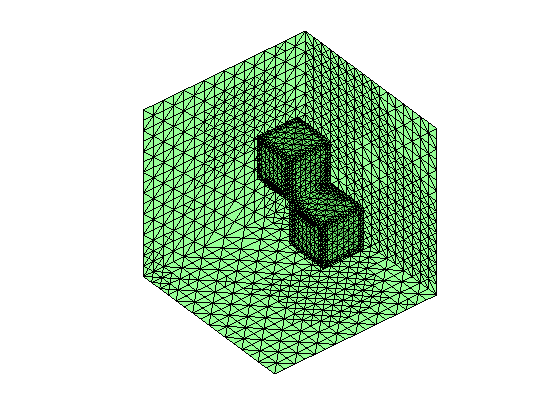}
\end{minipage}
}
\centering
\caption{The local refined mesh on adaptive levels 1 and 15}
\label{fig:5_13}
\end{figure*}

\begin{table}[H]
\centering
\begin{tabular}{|c|c|c|c|c|c|}
\hline
  Level($l$) & DOF. & Iters(Jacobi) & Iters(GS) & Res. & $\lambda_1^l$\\
  \hline
  4 & 1513 & 24 & 14 & 6.3401e-09 & 0.168228501\\
  \hline
  8 & 11731 & 26 & 14 & 8.1663e-09 & 0.189167802\\
  \hline
  11 & 47530 & 25 & 15 & 6.0027e-09 & 0.195310401\\

  \hline
  14 & 183585 & 25 & 15 & 5.1208e-09 & 0.198389974 \\
  \hline
  16 & 418433 & 25 & 14 & 9.5662e-09 & 0.199332014 \\
  \hline
  18 & 991722 & 26 & 15 & 7.7889e-09 & 0.199882133 \\
  \hline
\end{tabular}
\caption{The numerical results on each level}\label{tab:5_5}
\end{table}

\begin{figure}[htbp]
\begin{minipage}[t]{0.5\textwidth}
\centering
\includegraphics[width=\textwidth]{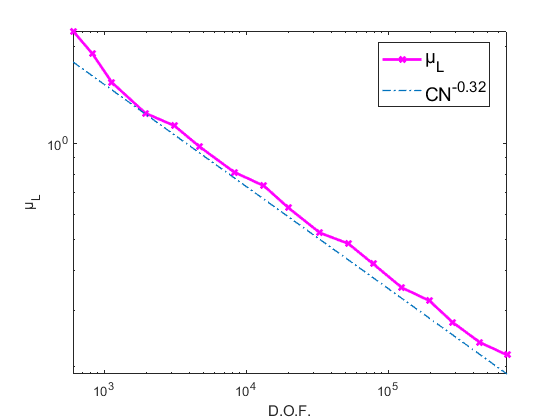}
\caption{Reduction of a posteriori error}
\label{fig:5_14}
\end{minipage}
\begin{minipage}[t]{0.5\textwidth}
\centering
\includegraphics[width=\textwidth]{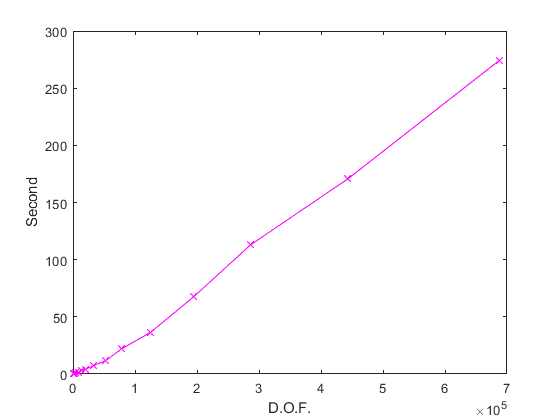}
\caption{Increasing of CPU time}
\label{fig:5_15}
\end{minipage}
\end{figure}

Figure \ref{fig:5_14} displays that the error curves obtained with uniform meshes and with adaptively refined meshes are almost parallel. Figure \ref{fig:5_15} shows that for the CPU time of the adaptive multilevel PHJD iteration is almost linear in terms of the degrees of freedom.

\begin{example}\label{ex_6}
    We consider the Maxwell eigenvalue problem on multiply
connected domain with cubic cavity: $\Omega = (0, 4\pi)^3 \setminus  (Q_1\cup Q_2)$, where $Q_1 = (2\pi,3\pi)\times(2\pi,3\pi)\times(0,4\pi)$ and $Q_2 = (0.5\pi,1.5\pi)\times(0.5\pi,1.5\pi)\times(1.5\pi,2.5\pi)$.
\end{example}

We modified the algorithm in the same manner as in Example \ref{ex_4}. Figure \ref{fig:6_1} depicts a locally refined mesh of 228,729 elements generated by the adaptive finite element algorithm. We observe that the mesh is locally refined near the boundary of the cavities and in the corner regions within the solenoid interior. It can be observed from Table \ref{tab:6_1} that in each case the iteration counts remain almost uniform on different levels, although the number of degrees of freedom increases up to 1,293,489. This confirms the quasi-optimality of the proposed method.

\begin{figure*}[htbp]
\centering
\subfigure[outside]{
\begin{minipage}[t]{0.45\linewidth} 
\centering
\includegraphics[scale=0.5]{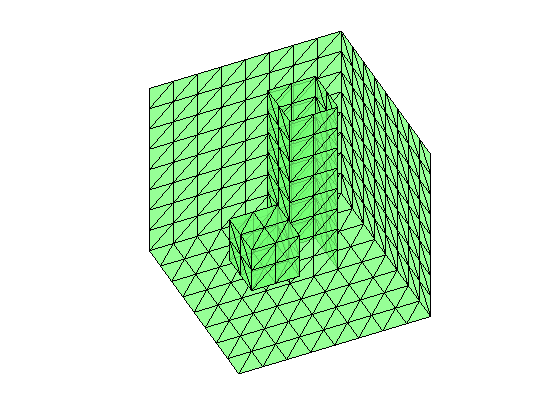} 
\end{minipage}
}
\subfigure[inside]{
\begin{minipage}[t]{0.45\linewidth} 
\centering
\includegraphics[scale=0.5]{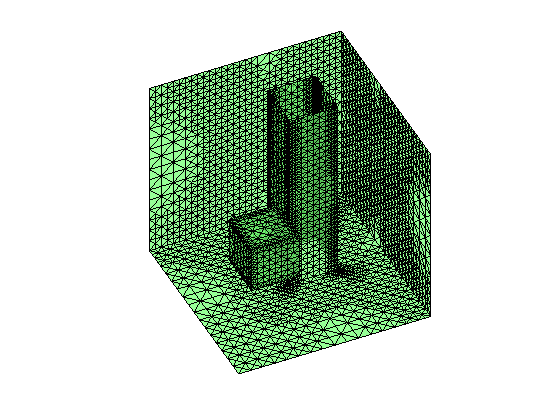}
\end{minipage}
}
\centering
\caption{The local refined mesh on adaptive level 11}
\label{fig:6_1}
\end{figure*}

\begin{table}[H]
\centering
\begin{tabular}{|c|c|c|c|c|c|}
\hline
  Level($l$) & DOF. & Iters(Jacobi) & Iters(GS) & Res. & $\lambda_1^l$\\
  \hline
  5 & 15728 & 23 & 14 & 4.6114e-09 & 0.059163470\\
  \hline
  7 & 37731 & 23 & 13 & 6.0091e-09 & 0.059657836\\
  \hline
  9 & 91931 & 22 & 13 & 5.5289e-09 & 0.059977658\\

  \hline
  11 & 228729 & 23 & 13 & 5.2923e-09 & 0.060170893 \\
  \hline
  13 & 534404 & 22 & 12 & 6.4284e-09 & 0.060282554 \\
  \hline
  15 & 1293489 & 22 & 12 & 7.8593e-09 & 0.060342783 \\
  \hline
\end{tabular}
\caption{The numerical results on each level}\label{tab:6_1}
\end{table}

\begin{figure}[htbp]
\begin{minipage}[t]{0.5\textwidth}
\centering
\includegraphics[width=\textwidth]{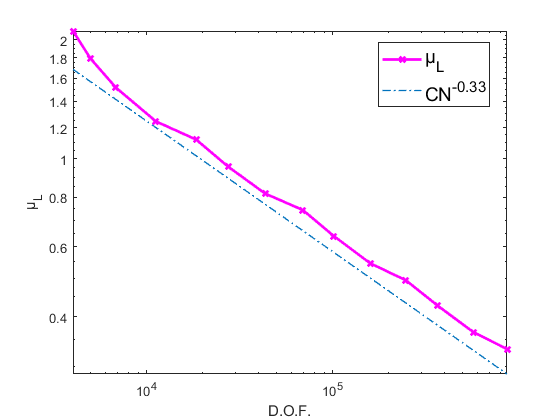}
\caption{Reduction of a posteriori error}
\label{fig:6_2}
\end{minipage}
\begin{minipage}[t]{0.5\textwidth}
\centering
\includegraphics[width=\textwidth]{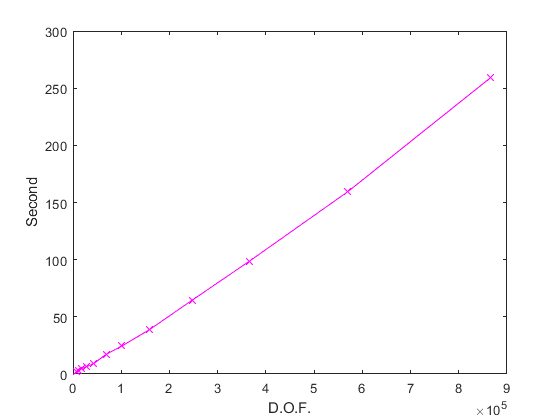}
\caption{Increasing of CPU time}
\label{fig:6_3}
\end{minipage}
\end{figure}

Figure \ref{fig:6_2} demonstrates the effectiveness and reliability of the a posteriori error estimator. Figure \ref{fig:6_3} shows that the CPU time of the adaptive multilevel PHJD iteration scales linearly with the number of degrees of freedom.

The next example bears out that the adaptive multilevel PHJD method is also efficient for the problems in non-Lipschitz domains, which are outside the scope of our theory.
\begin{example}\label{ex_5}
    We address the crack problem in its standard formulation. Considering the Maxwell eigenvalue problem on domain $\Omega = (0, 2\pi)^3 \setminus  (x,\pi,z):\pi \le x \le 2\pi, 0 \le z \le 2\pi $
\end{example}

\begin{figure*}[htbp]
\centering
\subfigure[outside]{
\begin{minipage}[t]{0.45\linewidth} 
\centering
\includegraphics[scale=0.5]{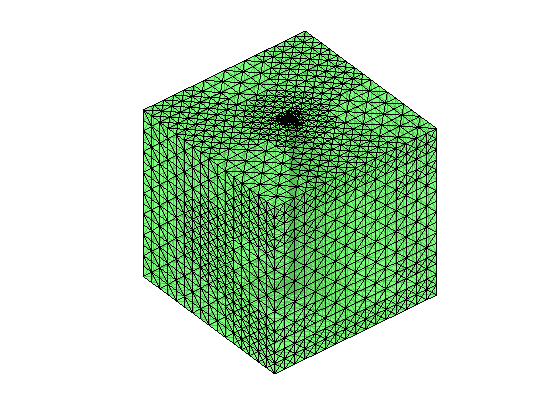} 
\end{minipage}
}
\subfigure[inside]{
\begin{minipage}[t]{0.45\linewidth} 
\centering
\includegraphics[scale=0.5]{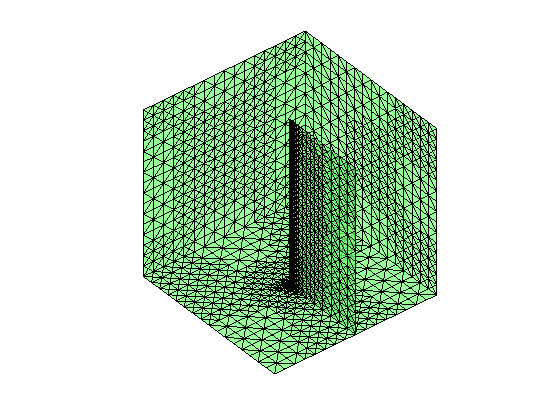}
\end{minipage}
}
\centering
\caption{The local refined mesh on adaptive level 20}
\label{fig:5_10}
\end{figure*}

Figure \ref{fig:5_10} displays a locally refined mesh of 233,311 elements generated by the adaptive finite element algorithm. It shows that singularities are predominantly localized in the central region of the crack. Table \ref{tab:5_4} indicates that the iteration count of the adaptive multilevel PHJD algorithm remains small and stable, even as the number of degrees of freedom increases up to 1,251,925. This indicates that the convergence rate is quasi-optimal, which illustrates that our adaptive multilevel PHJD method is scalable for the non-Lipschitz domain.

\begin{table}[H]
\centering
\begin{tabular}{|c|c|c|c|c|c|}
\hline
  Level($l$) & DOF. & Iters(Jacobi) & Iters(GS) & Res. & $\lambda_1^l$\\
  \hline
  5 & 285 & 24 & 16 & 5.6978e-09 & 0.319151975\\
  \hline
  9 & 2966 & 28 & 17 & 4.3451e-09 & 0.337844743\\
  \hline
  11 & 7989 & 29 & 17 & 4.4389e-09 & 0.345413445\\
  \hline
  15 & 40613 & 30 & 17 & 7.3772e-09 & 0.351146037\\
  \hline
  17 & 95097 & 29 & 16 & 5.8351e-09 & 0.352123877 \\
  \hline
  20 & 289815 & 29 & 16 & 4.8097e-09 & 0.353765744 \\
  \hline
  24 & 1251925 & 28 & 16 & 8.4293e-09 & 0.354253128 \\
  \hline
\end{tabular}
\caption{The numerical results on each level}\label{tab:5_4}
\end{table}

\begin{figure}[htbp]
\begin{minipage}[t]{0.5\textwidth}
\centering
\includegraphics[width=\textwidth]{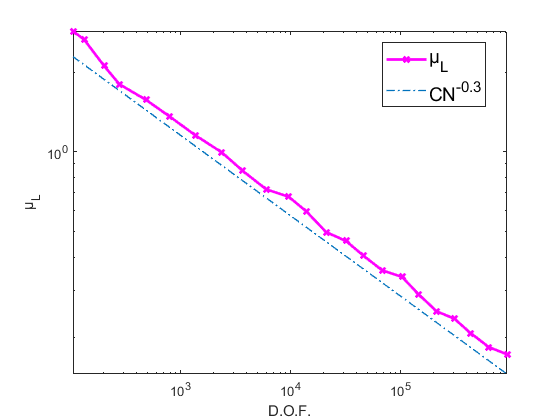}
\caption{Reduction of a posteriori error}
\label{fig:5_11}
\end{minipage}
\begin{minipage}[t]{0.5\textwidth}
\centering
\includegraphics[width=\textwidth]{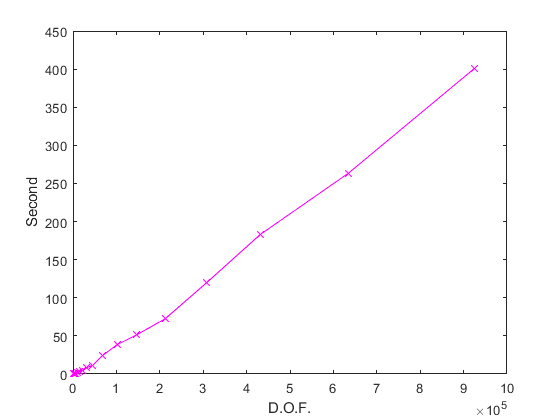}
\caption{Increasing of CPU time}
\label{fig:5_12}
\end{minipage}
\end{figure}

Due to the non-Lipschitz geometry, Figure \ref{fig:5_11} shows that the convergence rate of adaptively refined meshes 
deteriorates to a lower order compared to uniform meshes. Figure \ref{fig:5_12} illustrates that the CPU time of solving the algebraic system still increases roughly linearly with respect to the degrees of freedom, which reveals the quasi-optimality of the algorithm.

\appendix

\section{The proof of Lemma \ref{lem:4_10}}
\textbf{Proof of Lemma \ref{lem:4_10}}:\ For the estimate of $I_2^j$, we split it into two parts
\begin{equation*}
    I_{2,1}^j = \alpha^{j}Q_{2}^{L}B_{L}^{j}(A_{L}-\lambda^{j}I)Q_{1}^{L}u^{j},\ \ \ \ \ ~I_{2,2}^j = \alpha^{j}Q_{2}^{L}Q_{Uj}B_{L}^{j}r^{j}.
\end{equation*}
For $I_{2,1}^j$, using the definition of $\tilde{A}_L^j$, the Cauchy-Schwarz inequality, and the estimate $\|\tilde{E}_L^j\bm{v}\|_b\leq C\|\bm{v}\|_{E^j}$, we have
\begin{align*}
\|\alpha^{j}Q_{2}^{L}B_{L}^{j}\tilde{A}_{L}^jQ_{1}^{L}\bm{v}^{j}\|_{E^j}^2 
    =&~(\alpha^{j})^2(E_L^jQ_{1}^{L}\bm{u}^{j},Q_{2}^{L}E_L^jQ_{1}^{L}\bm{u}^{j})_{E^j}\\
    =&~(\alpha^{j})^2(Q_{1}^{L}\bm{u}^{j},\tilde{E}_L^jQ_{2}^{L}E_L^jQ_{1}^{L}\bm{u}^{j})_{E^j}\\
    \le &~C(\alpha^{j})^2(\lambda^j-\lambda_{1,L})\|Q_{1}^{L}\bm{u}^{j}\|_b\|Q_{2}^{L}E_L^jQ_{1}^{L}\bm{u}^{j}\|_{E^j}.
\end{align*}
Combining the above results with the estimate $\lambda^{j}- \lambda_{1,L} \leq C h_0^{2s}$ and the identity $\sqrt{\lambda^j-\lambda_{1,L}}\|Q_{1}^{L}\bm{u}^{j}\|_b = \|\bm{e}_L^j\|_{E^j}$, we obtain
\begin{equation*}
    \|I_{2,1}^j\|_{E^j}\le Ch_0^s\|\bm{e}_L^j\|_{E^j}.
\end{equation*}

Moreover, for $I_{2,2}^j$, using the definition of $Q_{U^j}$ and the Cauchy-Schwarz inequality, we get
    \begin{align*}
    \|\alpha^{j}Q_{2}^{L}Q_{U^j}B_{L}^{j}\bm{r}^{j}\|_{E^j} =~\alpha^{j}|b(B_L^j\bm{r}^j,\bm{u}^j)|\|Q_2^L\bm{u}^j\|_{E^j}
    \le ~\alpha^{j}\|B_L^j\bm{r}^j\|_b\|\bm{e}_L^j\|_{E^j}.
    \end{align*}
    To estimate $\|B_L^jr^j\|_b$, we utilize the definition of $E_L^j$ and the spatial decomposition to split the term into two parts
    \begin{align*}
    \|B_L^j\bm{r}^j\|_b 
    =~\|(I-E_L^j)\bm{u}^j\|_b
    \le ~\|(I-E_L^j)Q_1^L\bm{u}^j\|_b+\|(I-E_L^j)Q_2^L\bm{u}^j\|_b.
    \end{align*}
For the first part, applying \eqref{eq:3_2}, \eqref{eq:lem4_7}, and the definition of $T_0^j$, together with the estimates $\lambda^{j}- \lambda_{1,L} \leq C h_0^{2s}$ and $\|E_L^j Q_1^L\bm{u}^j\|_b \leq C\|Q_1^L\bm{u}^j\|_{a}$, yields
    \begin{align*}
    \|(I-E_L^j)Q_1^L\bm{u}^j\|_b\le&\|T_0^jQ_1^L\bm{u}^j\|_b+\|\sum_{l=1}^{L}T_l^jE_{l-1}^jQ_1^L\bm{u}^j\|_b\\
    \le&~ \|(\tilde{A}_0^j)^{-1}Q_2^0Q_0\tilde{A}_L^jQ_1^L\bm{u}^j\|_b+\|\sum_{l=1}^{L}T_l^jE_{l-1}^jQ_1^L\bm{u}^j\|_b\\
    \le &~C\sqrt{\lambda^j-\lambda_{1,L}}\|Q_1^L\bm{u}^j\|_{b}+Ch_0((Q_1^L\bm{u}^j,Q_1^L\bm{u}^j)_{E^j}\\
&+\lambda^jb(E_L^jQ_1^L\bm{u}^j,E_L^jQ_1^L\bm{u}^j)))^{\frac{1}{2}}\\
    \le &~ Ch_0^s.
    \end{align*}
   The second part is bounded using the triangle inequality, the estimates $\lambda^{j}- \lambda_{1,L} \leq C h_0^{2s}$ and the estimate of $\|E_L^jQ_2^L\bm{u}^j\|_b$. Specifically, we have
    \begin{align*}
    \|(I-E_L^j)Q_2^L\bm{u}^j\|_b\le \|\bm{e}_L^j\|_b+\|Q_2^L\bm{u}^j\|_a\le&~C \|\bm{e}_L^j\|_{E^j}\le Ch_0^s.
    \end{align*}
Combining the estimates of $I_{2,1}^{j}$ and $I_{2,2}^{j}$, we finish the proof. \qed


\bibliographystyle{plain}

\bibliography{references}

\end{document}